\documentclass{report}
\usepackage{amsmath,amssymb,titling,amsthm,enumitem,mathrsfs,euscript,upgreek,wasysym,tikz-cd,yfonts,hyperref,cleveref,graphicx,stmaryrd,bbold,yfonts}
\usepackage[utf8]{inputenc}
\usepackage[margin=1in]{geometry}
\usepackage[all,cmtip,2cell]{xy}
\UseAllTwocells
\xyoption{2cell}
\xyoption{rotate}

\usepackage{tikz}
\RequirePackage{tikz}

\pgfdeclarelayer{edgelayer}
\pgfdeclarelayer{nodelayer}
\pgfdeclarelayer{foreground}
\pgfdeclarelayer{background}
\pgfsetlayers{background,edgelayer,nodelayer,main,foreground,background}

\def\thickness{0.7pt}

\tikzstyle{dot}=[circle, draw=black, fill=black!25, inner sep=.4ex, line width=\thickness, node on layer=foreground]
\tikzstyle{blackdot}=[dot, fill=black!50]
\tikzstyle{blackdot}=[dot, fill=gray!40!white]
\tikzstyle{whitedot}=[dot, fill=white]
\tikzstyle{reddot}=[dot, fill=red]
\tikzstyle{greendot}=[dot, fill=green]

\makeatletter
\pgfkeys{%
  /tikz/on layer/.code={
    \pgfonlayer{#1}\begingroup
    \aftergroup\endpgfonlayer
    \aftergroup\endgroup
  },
  /tikz/node on layer/.code={
    \gdef\node@@on@layer{%
  \setbox\tikz@tempbox=\hbox\bgroup\pgfonlayer{#1}\unhbox\tikz@tempbox\endpgfonlayer\egroup}
\aftergroup\node@on@layer
  },
  /tikz/end node on layer/.code={
    \endpgfonlayer\endgroup\endgroup
  }
}
\def\node@on@layer{\aftergroup\node@@on@layer}
\makeatother

\newlength\morphismheight
\newlength\minimummorphismwidth
\newlength\stateheight
\newlength\minimumstatewidth
\newlength\connectheight
\tikzset{width/.initial=\minimummorphismwidth}

\newif\ifwedge\pgfkeys{/tikz/wedge/.is if=wedge}
\tikzset{wedge}
\newif\ifvflip\pgfkeys{/tikz/vflip/.is if=vflip}
\newif\ifhflip\pgfkeys{/tikz/hflip/.is if=hflip}
\newif\ifhvflip\pgfkeys{/tikz/hvflip/.is if=hvflip}
\newif\ifconnectnw\pgfkeys{/tikz/connect nw/.is if=connectnw}
\newif\ifconnectne\pgfkeys{/tikz/connect ne/.is if=connectne}
\newif\ifconnectsw\pgfkeys{/tikz/connect sw/.is if=connectsw}
\newif\ifconnectse\pgfkeys{/tikz/connect se/.is if=connectse}
\newif\ifconnectn\pgfkeys{/tikz/connect n/.is if=connectn}
\newif\ifconnects\pgfkeys{/tikz/connect s/.is if=connects}
\newif\ifconnectnwf\pgfkeys{/tikz/connect nw >/.is if=connectnwf}
\newif\ifconnectnef\pgfkeys{/tikz/connect ne >/.is if=connectnef}
\newif\ifconnectswf\pgfkeys{/tikz/connect sw >/.is if=connectswf}
\newif\ifconnectsef\pgfkeys{/tikz/connect se >/.is if=connectsef}
\newif\ifconnectnf\pgfkeys{/tikz/connect n >/.is if=connectnf}
\newif\ifconnectsf\pgfkeys{/tikz/connect s >/.is if=connectsf}
\newif\ifconnectnwr\pgfkeys{/tikz/connect nw </.is if=connectnwr}
\newif\ifconnectner\pgfkeys{/tikz/connect ne </.is if=connectner}
\newif\ifconnectswr\pgfkeys{/tikz/connect sw </.is if=connectswr}
\newif\ifconnectser\pgfkeys{/tikz/connect se </.is if=connectser}
\newif\ifconnectnr\pgfkeys{/tikz/connect n </.is if=connectnr}
\newif\ifconnectsr\pgfkeys{/tikz/connect s </.is if=connectsr}
\tikzset{keylengthnw/.initial=\connectheight}
\tikzset{keylengthn/.initial =\connectheight}
\tikzset{keylengthne/.initial=\connectheight}
\tikzset{keylengthsw/.initial=\connectheight}
\tikzset{keylengths/.initial =\connectheight}
\tikzset{keylengthse/.initial=\connectheight}
\tikzset{connect nw length/.style={connect nw=true, keylengthnw={#1}}}
\tikzset{connect n length/.style ={connect n =true, keylengthn ={#1}}}
\tikzset{connect ne length/.style={connect ne=true, keylengthne={#1}}}
\tikzset{connect sw length/.style={connect sw=true, keylengthsw={#1}}}
\tikzset{connect s length/.style ={connect s =true, keylengths ={#1}}}
\tikzset{connect se length/.style={connect se=true, keylengthse={#1}}}
\tikzset{connect nw < length/.style={connect nw <=true, keylengthnw={#1}}}
\tikzset{connect n < length/.style ={connect n <=true,  keylengthn ={#1}}}
\tikzset{connect ne < length/.style={connect ne <=true, keylengthne={#1}}}
\tikzset{connect sw < length/.style={connect sw <=true, keylengthnw={#1}}}
\tikzset{connect s < length/.style ={connect s <=true,  keylengths ={#1}}}
\tikzset{connect se < length/.style={connect se <=true, keylengthse={#1}}}
\tikzset{connect nw > length/.style={connect nw >=true, keylengthnw={#1}}}
\tikzset{connect n > length/.style ={connect n >=true,  keylengthn ={#1}}}
\tikzset{connect ne > length/.style={connect ne >=true, keylengthne={#1}}}
\tikzset{connect sw > length/.style={connect sw >=true, keylengthsw={#1}}}
\tikzset{connect s > length/.style ={connect s >=true,  keylengths ={#1}}}
\tikzset{connect se > length/.style={connect se >=true, keylengthse={#1}}}

\makeatletter
\pgfdeclareshape{morphism}
{
    \savedanchor\centerpoint
    {
        \pgf@x=0pt
        \pgf@y=0pt
    }
    \anchor{center}{\centerpoint}
    \anchorborder{\centerpoint}
    \saveddimen\savedlengthnw
    {
        \pgfkeysgetvalue{/tikz/keylengthnw}{\len}
        \pgf@x=\len
    }
    \saveddimen\savedlengthn
    {
        \pgfkeysgetvalue{/tikz/keylengthn}{\len}
        \pgf@x=\len
    }
    \saveddimen\savedlengthne
    {
        \pgfkeysgetvalue{/tikz/keylengthne}{\len}
        \pgf@x=\len
    }
    \saveddimen\savedlengthsw
    {
        \pgfkeysgetvalue{/tikz/keylengthsw}{\len}
        \pgf@x=\len
    }
    \saveddimen\savedlengths
    {
        \pgfkeysgetvalue{/tikz/keylengths}{\len}
        \pgf@x=\len
    }
    \saveddimen\savedlengthse
    {
        \pgfkeysgetvalue{/tikz/keylengthse}{\len}
        \pgf@x=\len
    }
    \saveddimen\overallwidth
    {
        \pgfkeysgetvalue{/tikz/width}{\minwidth}
        \pgf@x=\wd\pgfnodeparttextbox
        \ifdim\pgf@x<\minwidth
            \pgf@x=\minwidth
        \fi
    }
    \savedanchor{\upperrightcorner}
    {
        \pgf@y=.5\ht\pgfnodeparttextbox
        \advance\pgf@y by -.5\dp\pgfnodeparttextbox
        \pgf@x=.5\wd\pgfnodeparttextbox
    }
    \anchor{north}
    {
        \pgf@x=0pt
        \pgf@y=0.5\morphismheight
    }
    \anchor{north east}
    {
        \pgf@x=\overallwidth
        \multiply \pgf@x by 2
        \divide \pgf@x by 5
        \pgf@y=0.5\morphismheight
    }
    \anchor{east}
    {
        \pgf@x=\overallwidth
        \divide \pgf@x by 2
        \advance \pgf@x by 5pt
        \pgf@y=0pt
    }
    \anchor{west}
    {
        \pgf@x=-\overallwidth
        \divide \pgf@x by 2
        \advance \pgf@x by -5pt
        \pgf@y=0pt
    }
    \anchor{north west}
    {
        \pgf@x=-\overallwidth
        \multiply \pgf@x by 2
        \divide \pgf@x by 5
        \pgf@y=0.5\morphismheight
    }
    \anchor{connect nw}
    {
        \pgf@x=-\overallwidth
        \multiply \pgf@x by 2
        \divide \pgf@x by 5
        \pgf@y=0.5\morphismheight
        \advance\pgf@y by \savedlengthnw
    }
    \anchor{connect ne}
    {
        \pgf@x=\overallwidth
        \multiply \pgf@x by 2
        \divide \pgf@x by 5
        \pgf@y=0.5\morphismheight
        \advance\pgf@y by \savedlengthne
    }
    \anchor{connect sw}
    {
        \pgf@x=-\overallwidth
        \multiply \pgf@x by 2
        \divide \pgf@x by 5
        \pgf@y=-0.5\morphismheight
        \advance\pgf@y by -\savedlengthsw
    }
    \anchor{connect se}
    {
        \pgf@x=\overallwidth
        \multiply \pgf@x by 2
        \divide \pgf@x by 5
        \pgf@y=-0.5\morphismheight
        \advance\pgf@y by -\savedlengthse
    }
    \anchor{connect n}
    {
        \pgf@x=0pt
        \pgf@y=0.5\morphismheight
        \advance\pgf@y by \savedlengthn
    }
    \anchor{connect s}
    {
        \pgf@x=0pt
        \pgf@y=-0.5\morphismheight
        \advance\pgf@y by -\savedlengths
    }
    \anchor{south east}
    {
        \pgf@x=\overallwidth
        \multiply \pgf@x by 2
        \divide \pgf@x by 5
        \pgf@y=-0.5\morphismheight
    }
    \anchor{south west}
    {
        \pgf@x=-\overallwidth
        \multiply \pgf@x by 2
        \divide \pgf@x by 5
        \pgf@y=-0.5\morphismheight
    }
    \anchor{south}
    {
        \pgf@x=0pt
        \pgf@y=-0.5\morphismheight
    }
    \anchor{text}
    {
        \upperrightcorner
        \pgf@x=-\pgf@x
        \pgf@y=-\pgf@y
    }
    \backgroundpath
    {
        \pgfsetstrokecolor{black}
        \pgfsetlinewidth{\thickness}
        \begin{scope}
                \ifhflip
                    \pgftransformyscale{-1}
                \fi
                \ifvflip
                    \pgftransformxscale{-1}
                \fi
                \ifhvflip
                    \pgftransformxscale{-1}
                    \pgftransformyscale{-1}
                \fi
                \pgfpathmoveto{\pgfpoint
                    {-0.5*\overallwidth-5pt}
                    {0.5*\morphismheight}}
                \pgfpathlineto{\pgfpoint
                    {0.5*\overallwidth+5pt}
                    {0.5*\morphismheight}}
                \ifwedge
                    \pgfpathlineto{\pgfpoint
                        {0.5*\overallwidth + 15pt}
                        {-0.5*\morphismheight}}
                \else
                    \pgfpathlineto{\pgfpoint
                        {0.5*\overallwidth + 5pt}
                        {-0.5*\morphismheight}}
                \fi
                \pgfpathlineto{\pgfpoint
                    {-0.5*\overallwidth-5pt}
                    {-0.5*\morphismheight}}
                \pgfpathclose
                \pgfusepath{stroke}
        \end{scope}
        \ifconnectnw
            \pgfpathmoveto{\pgfpoint
                {-0.4*\overallwidth}
                {0.5*\morphismheight}}
            \pgfpathlineto{\pgfpoint
                {-0.4*\overallwidth}
                {0.5*\morphismheight+\savedlengthnw}}
            \pgfusepath{stroke}
        \fi
        \ifconnectne
            \pgfpathmoveto{\pgfpoint
                {0.4*\overallwidth}
                {0.5*\morphismheight}}
            \pgfpathlineto{\pgfpoint
                {0.4*\overallwidth}
                {0.5*\morphismheight+\savedlengthne}}
            \pgfusepath{stroke}
        \fi
        \ifconnectsw
            \pgfpathmoveto{\pgfpoint
                {-0.4*\overallwidth}
                {-0.5*\morphismheight}}
            \pgfpathlineto{\pgfpoint
                {-0.4*\overallwidth}
                {-0.5*\morphismheight-\savedlengthsw}}
            \pgfusepath{stroke}
        \fi
        \ifconnectse
            \pgfpathmoveto{\pgfpoint
                {0.4*\overallwidth}
                {-0.5*\morphismheight}}
            \pgfpathlineto{\pgfpoint
                {0.4*\overallwidth}
                {-0.5*\morphismheight-\savedlengthse}}
            \pgfusepath{stroke}
        \fi
        \ifconnectn
            \pgfpathmoveto{\pgfpoint
                {0pt}
                {0.5*\morphismheight}}
            \pgfpathlineto{\pgfpoint
                {0pt}
                {0.5*\morphismheight+\savedlengthn}}
            \pgfusepath{stroke}
        \fi
        \ifconnects
            \pgfpathmoveto{\pgfpoint
                {0pt}
                {-0.5*\morphismheight}}
            \pgfpathlineto{\pgfpoint
                {0pt}
                {-0.5*\morphismheight-\savedlengths}}
            \pgfusepath{stroke}
        \fi
        \ifconnectnwf
            \draw [forward arrow style] (-0.4*\overallwidth,0.5*\morphismheight)
                to (-0.4*\overallwidth,0.5*\morphismheight+\savedlengthnw);
        \fi
        \ifconnectnef
            \draw [forward arrow style] (0.4*\overallwidth,0.5*\morphismheight)
                to (0.4*\overallwidth,0.5*\morphismheight+\savedlengthne);
        \fi
        \ifconnectswf
            \draw [forward arrow style] (-0.4*\overallwidth,-0.5*\morphismheight-\savedlengthsw)
                to (-0.4*\overallwidth,-0.5*\morphismheight);
        \fi
        \ifconnectsef
            \draw [forward arrow style] (0.4*\overallwidth,-0.5*\morphismheight-\savedlengthse)
                to (0.4*\overallwidth,-0.5*\morphismheight);
    \fi
    \ifconnectnf
        \draw [forward arrow style] (0,0.5*\morphismheight)
            to (0,0.5*\morphismheight+\savedlengthn);
    \fi
    \ifconnectsf
        \draw [forward arrow style] (0,-0.5*\morphismheight-\savedlengths)
            to (0,-0.5*\morphismheight);
    \fi
    \ifconnectnwr
        \draw [reverse arrow style] (-0.4*\overallwidth,0.5*\morphismheight)
            to (-0.4*\overallwidth,0.5*\morphismheight+\savedlengthnw);
    \fi
    \ifconnectner
        \draw [reverse arrow style] (0.4*\overallwidth,0.5*\morphismheight)
            to (0.4*\overallwidth,0.5*\morphismheight+\savedlengthne);
    \fi
    \ifconnectswr
        \draw [reverse arrow style] (-0.4*\overallwidth,-0.5*\morphismheight-\savedlengthsw)
            to (-0.4*\overallwidth,-0.5*\morphismheight);
    \fi
    \ifconnectser
        \draw [reverse arrow style] (0.4*\overallwidth,-0.5*\morphismheight-\savedlengthse)
            to (0.4*\overallwidth,-0.5*\morphismheight);
    \fi
    \ifconnectnr
        \draw [reverse arrow style] (0,0.5*\morphismheight)
            to (0,0.5*\morphismheight+\savedlengthn);
    \fi
    \ifconnectsr
        \draw [reverse arrow style] (0,-0.5*\morphismheight-\savedlengths)
            to (0,-0.5*\morphismheight);
    \fi
}
}
\makeatother

\pagecolor{white}

\xymatrixrowsep{20mm}
\xymatrixcolsep{20mm}

\theoremstyle{plain}
\newtheorem{thm}{Theorem}[section]
\newtheorem{lem}[thm]{Lemma}

\makeatletter
\newcommand\varitem[1]{\bfem[\textbf{A\arabic{enumi}\rlap{$#1$}.}]%
  \edef\@currentlabel{A\arabic{enumi}{$#1$}}}
\makeatother

\crefname{lemma}{Lemma}{Lemmas}

\makeatletter
\newcommand{\icirc}{\mathbin{\mathpalette\make@small\oplus}}
\newcommand{\smallotimes}{\mathbin{\mathpalette\make@small\otimes}}

\newcommand{\make@small}[2]{%
  \vcenter{\hbox{%
    \scalebox{0.6}{$\m@th#1#2$}%
  }}%
}
\makeatother

\theoremstyle{definition}
\newtheorem{defn}{Definition}[section]

\theoremstyle{remark}

\DeclareFontFamily{U}{skulls}{}
\DeclareFontShape{U}{skulls}{m}{n}{ <-> skull }{}

\DeclareFontFamily{U}{min}{}
\DeclareFontShape{U}{min}{m}{n}{<-> udmj30}{}

\begin{document}

{\centering\scshape\Large\textsc{An Axiomatic Approach to Higher Order Set Theory} \par}
{\centering\scshape\textsc{Alec Rhea} \par}

\begin{abstract}
Higher order set theory has been a topic of interest for some time, with recent efforts focused on the strength of second order set theories [Kam19]. In this paper we strive to present one 'theory of collections' that allows for a formal consideration of 'countable higher order set theory'. We will see that this theory is equiconsistent with $ZFC$ plus the existence of a countable collection of inaccessible cardinals. We will also see that this theory serves as a canonical foundation for some parts of mathematics not covered by standard set/class theories (e.g. $ZFC$ or $MK$), such as category theory.  \\
\end{abstract}

\tableofcontents

\chapter{Higher Order Set Theory}

Here we lay out the primitive notions under consideration, the language we will express them in, and the axioms these primitive notions obey. \\

\section{Primitive Notions}

The language is the first order language of set theory. The primitives are {\bf collections}, denoted with capital letters from the end of the alphabet $X,Y,Z,\dots$, together with {\bf collection membership}, denoted by $\in$. We may have that a collection $X$ {\bf is a member} of a collection $Y$, denoted $X\in Y$, or not, denoted $X\notin Y$. We have countable list of specified collections denoted with indexed individual constants $\{\mathcal{C}_n\}_{n<\omega}$, each called the {\bf $n+1$ collection of $n$-collections}, and members $X\in\mathcal{C}_n$ are called {\bf $n$-collections}. For $n>0$ we call an $n$-collection $X\in\mathcal{C}_n$ {\bf proper} iff it isn't an $m$-collection for any $m<n$. We will call $0$-collections {\bf sets}. \\

\section{Axioms}

The axioms are as follows. We have all the axioms of Zermelo set theory $Z$ [Zer08] plus foundation (denoted $Z^+$) for collections, and each $\mathcal{C}_n$ models all of $Z^+$ as well. Explicitly, recall that $Z^+$ can be axiomatized as follows (where we obtain the original phrasing by replacing the word 'collection' with the word 'set' and replacing uppercase letters $X,Y,Z,\dots$ denoting collections with lowercase letters $x,y,z,\dots$ denoting sets): \\
 
\subsubsection{Axioms of $Z,Z^+$}
\begin{itemize}
\item {\bf Z1 (Extensionality)}. Two collections are equal iff they have exactly the same members. $$\forall X\forall Y\big(X=Y\iff\forall Z(Z\in X\iff Z\in Y)\big).$$ For collections $X$ and $Y$ we define $$X\subseteq Y\iff\forall Z(Z\in X\implies Z\in Y)$$ and say that {\bf $Y$ contains $X$} in this situation, or that {\bf $X$ is a subcollection of $Y$}, or that {\bf $Y$ is a supercollection of $X$}.
\item {\bf Z2 (Separation)}. For any predicate $\phi$ and collection $X$ there exists a collection $Y$ whose members are precisely the members of $X$ satisfying $\phi$. $$\forall\phi\forall X\exists Y\forall Z\big(Z\in Y\iff Z\in X\wedge\phi(Z)\big).$$ We denote the collection $Y$ guaranteed by the instance of this axiom schemata at a predicate $\phi$ and collection $X$ by $$\{Z\in X|\phi(Z)\},$$ read as {\it the collection of all $Z$ in $X$ such that $\phi(Z)$}. We define the {\bf empty collection} $$\emptyset=\{X\in\mathcal{C}_0|X\neq X\},$$ which is unique by extensionality.
\item {\bf Z3 (Pairing)}. For any collections $X$ and $Y$, there exists a collection $Z$ whose members are precisely $X$ and $Y$. $$\forall X\forall Y\exists Z\forall A(A\in Z\iff A=X\vee A=Y).$$ We denote the collection $Z$ guaranteed by the instance of this axiom at collections $X$ and $Y$ by $$\{X,Y\}.$$ For any collection $X$ we define the {\bf singleton} $$\{X\}=\{X,X\},$$ and for any collections $X,Y$ we define the {\bf ordered pair} $$(X,Y)=\{\{X,\},\{X,Y\}\},$$ referring to $X$ and $Y$ as the {\bf first} and {\bf second coordinates} of $(X,Y)$, respectively.
\item {\bf Z4 (Union)}. For any collection $X$, there exists a collection $Y$ whose members are precisely the members of members of $X$. $$\forall X\exists Y\forall Z\big(Z\in Y\iff \exists A\in X(Z\in A)\big).$$ We denote the collection $Y$ guaranteed by this axiom at a collection $X$ by $$\bigcup X.$$ For collections $X$ and $Y$ we define the {\bf binary union} $$X\cup Y=\bigcup\{X,Y\},$$ and the {\bf successor} $$\mathcal{S}(X)=X\cup\{X\}.$$ We denote by $\mathcal{S}^n(X)$ the collection obtained by taking a collection $X$ and applying the successor operation $n$ times iteratively, so $$\mathcal{S}^n(X)=X\cup\{X\}\cup\{X\cup\{X\}\}\cup\cdots$$
\item {\bf Z5 (Infinity)}. There exists a collection containing exactly the empty collection and all of its successors. $$\exists X\Big(\emptyset\in X\wedge\forall Y(Y\in X\iff\exists Z\in X(Y=\mathcal{S}Z)\big)\Big).$$ We denote the collection $X$ whose existence is guaranteed by this axiom by $$\omega$$ and its uniqueness satisfying this axiom is a consequence of extensionality.
\item {\bf Z6 (Powercollection)}. For any collection $X$ there exists a collection $Y$ whose members are precisely the subcollections of $X$. $$\forall X\exists Y\forall Z(Z\subseteq X\iff Z\in Y).$$ We denote the collection $Y$ guaranteed by this axiom together with a collection $X$ by $$\mathcal{P}(X).$$ For collections $X$ and $Y$, we define the {\bf Cartesian product of $X$ and $Y$} to be the collection $$X\times Y=\{(X',Y')\in\mathcal{P}(\mathcal{P}(X\cup Y))|X'\in X\wedge Y'\in Y\}.$$ A {\bf relation from $X$ to $Y$} is a subcollection $R\subseteq X\times Y$. We say that a relation $R$ from $X$ to $Y$ is {\bf entire} iff every member of $X$ appears at least once as a first coordinate in $R$, and {\bf functional} iff each member of $X$ appears at most once as a first coordinate in $R$. A {\bf function from $X$ to $Y$} is a relation from $X$ to $Y$ which is entire and functional; we call $X$ the {\bf domain} and $Y$ the {\bf codomain} of the function. We denote a function $F$ from $X$ to $Y$ by $$F:X\to Y.$$ A {\bf relation} is a collection that is a relation from $X$ to $Y$ for some collections $X,Y$, and a function is a relation that is entire and functional. For a function $F$, we denote by $dmnF$ the collection of first coordinates in $F$ and $rngF$ the collection of second coordinates in $F$. We denote a function $F$ with domain $X$ in {\bf function notation} by $$F=\langle F(X'):X'\in X\rangle.$$ 
\item {\bf Z7 (Choice)}. For any collection $X$ not containing the empty collection, there exists a choice function on $X$. $$\forall X\Big(\emptyset\notin X\implies\exists F\big(F\ \text{is a function}\wedge dmnF=X\wedge rngF\subseteq\bigcup X\wedge \forall Y\in X(F(Y)\in Y)\big)\Big).$$
\end{itemize}
We define $Z$ to be the concatenation of these axioms. That is, $$Z=\bigwedge_{0<i<8}{\bf Z}i.$$ We obtain $Z^+$ from $Z$ by adding
\begin{itemize}
\item {\bf F1 (Foundation)}. All collections contain a member they are disjoint from. $$\forall X\exists Y(Y\in X\wedge Y\cap X=\emptyset).$$ 
\end{itemize}
So $$Z^+=Z\wedge{\bf F1}.$$ \\

\begin{itemize}
\item {\bf A1 - $Z^+$ for Collections}. All of the above are axioms in this theory. \\
\end{itemize}

\begin{defn}[{\it Models, Elementary Submodel}]
Let $\phi$ be a predicate in the language of collection theory, with $X$ a collection. We define the {\bf relativization of $\phi$ to $X$} denoted $\phi^X$, to be the predicate obtained from $\phi$ by binding all existential/universal quantifiers and free variables in $\phi$ to $X$. We say that $X$ {\bf models $\phi$}, written $$X\models\phi,$$ iff relativizing $\phi$ to $X$ produces a true statement. $$X\models\phi\iff\phi^X=\top.$$ We will say that a collection $X$ is an {\bf elementary submodel} of a collection $Y$, denoted $$X\preceq Y,$$ iff the following two conditions hold:
\begin{enumerate}
\item $X\subseteq Y$.
\item $Y$ doesn't model anything about members of $X$ that $X$ doesn't model. That is, for any collection $\{X_\alpha\}_{i<n}\subseteq X$ and any $n$-ary predicate $\phi$, we have that $$Y\models\phi(\{X_i\}_{i<n})\iff X\models\phi(\{X_i\}_{i<n}).$$ \\
\end{enumerate}
\end{defn}

So, for example, the relativization of pairing ({\bf Z3}) to a collection $X$ is the predicate $${\bf Z3}^X=\forall x\in X\forall y\in X\exists z\in X\forall a\in X(a\in z\iff a=x\vee a=y),$$ and $X\models{\bf Z3}$ iff this statement is true, so arbitrary doubletons of members in $X$ are also members of $X$. \\

\begin{defn}
We will say that a collection $X$ is {\bf transitive} iff members of members of $X$ are also members of $X$. $$X\ \text{is transitive}\iff\forall Y\forall Z(Y\in Z\in X\implies Y\in X).$$ We say that $X$ is {\bf complete} iff $X$ is transitive and subcollections of members of $X$ are also members of $X$. $$X\ \text{is complete}\iff X\ \text{is transitive}\wedge\forall Y\forall Z(Y\subseteq Z\in X\implies Y\in X).$$ \\
\end{defn}

\begin{itemize}
\item {\bf A2 - Complete Structured Collections of Collections}. Each collection of collections is complete and models Zermelo set theory plus foundation. $$\forall n<\omega(\mathcal{C}_n\ \text{is complete}\wedge\mathcal{C}_n\models Z^+).$$
\item {\bf A3 - Collection Hierarchy}. For all $n<\omega$, $\mathcal{C}_n$ is an elementary submodel of $\mathcal{C}_{n+1}$. $$\forall n<\omega(\mathcal{C}_n\preceq\mathcal{C}_{n+1}).$$ \\
\end{itemize}

This allows us to do 'all the basics' we expect to be able to do with each collection of collections, and ensures that each stage behaves well with previous stages. So far, this theory is (signifigantly) below $ZFC$ in consistency strength since $V_{\omega^2}$ is a model of this theory with $\mathcal{C}_n=V_{\omega\cdot(n+2)}$ for all $n<\omega$, where $V_\alpha$ denotes the $\alpha^{th}$ stage of the cumulative hierarchy in $ZFC$. Now for the big guns -- up first, we add the ability to freely define $n+1$-collections of $n$-collections. For all $n<\omega$, say that a predicate $\phi$ in the language of collection theory is {\bf safe above $n$} iff no $m$-collections for $m>n$ occur in $\phi$, and denote by $\Phi_n$ the collection of all predicates safe above $n$. \\

\begin{itemize}
\item {\bf A4 - Collection Building Axioms}. For any predicate $\phi$ safe above $n$, there exists an $n+1$ collection whose members are precisely the $n$-collections satisfying $\phi$. $$\forall\phi\in\Phi_n\exists X\in\mathcal{C}_{n+1}\forall Y\big(Y\in X\iff Y\in\mathcal{C}_n\wedge\phi(Y)\big).$$ We denote the collection $X$ guaranteed by the instance of this axiom at a predicate $\phi$ safe above $n$ by $${^{n+1}}\{X|\phi(X)\},$$ read as {\it the $n+1$-collection of $n$-collections $X$ such that $\phi(X)$}. We omit the superscript when it is obvious from context. \\
\end{itemize}

\begin{defn}[{\it Universe of $n$-Collections}]
For each $n<\omega$, we define $$\widehat{V_n}={^{n+1}}\{X|X=X\}\in\mathcal{C}_{n+1}.$$ We will refer to $\widehat{V_n}$ as the {\bf universe of $n$-collections}.
\end{defn}
\begin{lem}
$\forall n<\omega(\widehat{V_n}=\mathcal{C}_n)$. \\
\end{lem}

The final axiom below gives us replacement in each collection of collections, greatly strengthening the consistency strength of the theory. First, note that we can pair/union/etc. $n$-collections and $m$-collections for $m<n$ since $m$-collections are $n$-collections in this situation by {\bf A3}, and all collection stages support the standard operations by {\bf A2}.  \\

\begin{itemize}
\item {\bf A5 - Replacement}. For all $n<\omega$, if $F$ is a function and $dmnF$ is an $n$-collection and $F(X)$ is an $n$-collection for all $X\in dmnF$, then $rngF$ is an $n$-collection. $$\forall F\Big(F\ \text{is a function}\wedge dmnF\in\widehat{V_n}\wedge\forall X\in dmnF(F(X)\in\widehat{V_n})\implies rngF\in\widehat{V_n}\Big).$$ \\
\end{itemize}

These axioms are inspired by the axioms of class building and replacement in Donald Monks concise treatment of MK class theory [Mon69]. Note that collection building and replacement together trivially imply replacement (in the $ZFC$ sense) for sets (and $n$-collections for all $n<\omega$); for any binary predicate $\phi(A,B)$ and set $x$ such that for each $x'\in x$ there exists some unique set $y$ with $\phi(x',y)=\top$, define $$Y={^1}\big\{y:\exists x'\in x\big(\phi(x',y)\big)\big\}$$ and for all $x'\in x$ denote by $\overline\phi(x')$ the unique set $y$ such that $\phi(x,y)$.  Then $$F=\langle\overline\phi(x'):x'\in x\rangle\subseteq x\times Y$$ is a function by uniqueness of $\overline\phi(x)$ which is further surjective by the definition of $Y$, $x$ is a set by assumption, and $F(x')=\overline\phi(x)$ is also a set for all $x'\in x$ by assumption, so $rngF=Y$ is also a set by replacement; we thusly have all of standard $ZFC$ for sets in this theory. This will be a somewhat standard feature of higher order set theory -- we can define any desired $n+1$-collection of $n$-collections with reckless abandon, but have to more judiciously apply the axioms of $Z^+$ together with replacement at each stage to show that the $n+1$ collections we define via collection building are sometimes also $n$-collections.  \\

\section{Consistency Strength}

This theory proves that $\mathcal{C}_0=\widehat{V_0}$ is a model of $ZFC$ and so exceeds $ZFC$ in consistency strength. On the other hand, the theory $ZFC$ plus the existence of a countable collection $\{\kappa_n\}_{n<\omega}$ of strictly inaccessible cardinals with $\kappa_n<\kappa_{n+1}$ for all $n<\omega$ can model this theory by setting $\mathcal{C}_n=V_{\kappa_n}$ for all $n<\omega$, and so serves as an upper bound on its consistency strength. The author conjectures that the latter is also a lower bound on the consistency strength of this theory. For a proof sketch, define ordinals in this theory to be hereditarily membership transitive collections, define rank $\rho$ as usual for all collections, and for each $n<\omega$ define $$Ord_n={^n}\{X:X\ \text{is an ordinal}\}.$$ In particular, each $Ord_n$ is a proper $n+1$ collection since it is also an ordinal and would thusly be a member of itself by definition if it were an $n$-collection, contradicting foundation. Further, defining stages of the cumulative hierarchy $V_\alpha$ for ordinals $\alpha$ using rank in the usual way we have that $$V_{Ord_n+1}\models MK$$ for all $n<\omega$. To see this, observe that for all $n<\omega$ we have that  $$V_{Ord_n}=\widehat{V_n}\models ZFC$$ and if we have two $n+1$-collections $X,Y\in V_{Ord_n+1}$ (viewed as 'classes over $V_{Ord_n}$') with $X\in Y$ we have $$\rho(X)<\rho(Y)\leq Ord_n\implies X\in V_{Ord_n}=\widehat{V_n},$$ so a 'class $X$ over $V_{Ord_n}$' becomes a 'set in $V_{Ord_n}$' as soon as it is a member of another 'class over $V_{Ord_n}$'. Using the characterization of strictly inaccessible cardinals as those ordinals $\alpha$ such that $V_{\alpha+1}\models MK$, we see that $Ord_n$ is strictly inaccessible for all $n<\omega$. \\

\section{Potentialist vs. Completionist Set Theory}

There is a philosophical debate raging vaguely along the lines of wether we should view things as 'completed' or 'unending', from infinity to sets themselves, where 'potentialists' want things to be forever iterable and view references to 'all the collections' as a 'completed object' to be bad pool. Completionists feel the opposite, that it is coherent to refer to 'all the collections' as an object of discourse. The theory as presented above is potentialist in nature -- we can collect things up, and collect up those collections, so on and so forth, but have no ability to refer to 'all collections at all stages'. We could turn this into a theory satisfying completionists by adding one final putative axiom. \\

\begin{itemize}
\item {\bf A6 - Collection Building Completion}. For any predicate $\phi$ there exists a collection $X$ whose members are precisely the $n$-collections for all $n<\omega$ satisfying $\phi$. $$\forall\phi\exists X\forall Y\Big(Y\in X\iff\exists n\big(Y\in\mathcal{C}_n\wedge\phi(Y)\big)\Big).$$ We denote the collection $X$ guaranteed by this axiom together with a predicate $\phi$ by $$^{\omega}\{Y|\phi(Y)\},$$ referring to $X$ as an {\bf $\omega$-collection} and omitting the superscript when it is obvious from context. If $X$ is an $\omega$-collection that is not an $n$-collection for any $n$, we call $X$ a {\bf proper $\omega$-collection}. \\
\end{itemize}

This axiom allows us to 'close' the iterative collection forming process and make reference to a collection of 'all collections (that we care about)' $$\widehat{V_\omega}={^\omega}\{X:X=X\},$$ which is also provably not an $n$-collection for any $n<\omega$, so proper $\omega$-collections exist as soon as this axiom is added. Note that $\widehat{V_\omega}$ is also a model of this theory without {\bf A6} and thusly proves the consistency of the theory laid out in the preceding section, but this is the only increase in consistency we get from this additional axiom. We will refer to the theory given by the primitives and language above together with {\bf A1}$-${\bf A5} as {\it higher order set theory} ($HOST$), and the theory obtained by adding {\bf A6} as {\it generalized higher order set theory} ($GHOST$). As a final note, Cantor's theorem in $ZFC$ on the lack of a surjection from a set to its powerset generalizes trivially to the following in $HOST$. \\

\begin{lem}[{\it Cantor}]
Let $X$ be a collection. Then $|X|<|\mathcal{P}(X)|$.
\end{lem}
\begin{proof}
Diagonalization as usual, interpreting $|X|<|\mathcal{P}(X)|$ to mean 'there is no surjection $s:X\to\mathcal{P}(X)$'. \\
\end{proof}

\chapter{Mathematics in Higher Order Set Theory}

In this section, we will see that $HOST$ frees us from several standard 'constraints' in set and class theories like $ZFC/GBC/MK/etc.$  We do not assume {\bf A6} in what follows unless otherwise specified.\\

\section{Categories Large and Small}

In almost any treatment of category theory, there is some discussion of which categories are 'large' and which categories are 'small' in a given foundation. In $HOST$, we have a hierarchy of notions of small/large-ness. \\

\begin{defn}[{\it $n$-Large Categories}]
A {\bf category} $\mathcal{C}$ consists of the following data:
\begin{enumerate}
\item A collection of {\bf objects}, denoted $${\bf Ob}_\mathcal{C}.$$
\item A collection of {\bf arrows}, denoted $${\bf Hom}_\mathcal{C}.$$
\item {\bf Domain} and {\bf codomain} selecting functions $${\sf dom,cod}:{\bf Hom}_\mathcal{C}\rightrightarrows{\bf Ob}_\mathcal{C}.$$ For each pair of objects $X,Y\in{\bf Ob}_\mathcal{C}$ we define a collection of {\bf arrows from $X$ to $Y$}, denoted ${\bf Hom}_\mathcal{C}(X,Y)$, by $${\bf Hom}_\mathcal{C}(X,Y)=\{f\in{\bf Hom}_\mathcal{C}|{\sf dom}(f)=X\wedge{\sf cod}(f)=Y\}.$$ We denote an arrow from $X$ to $Y$ by $$f:X\to Y.$$
\item For each object $X\in{\bf Ob}_\mathcal{C}$, an {\bf identity selecting function} $$1_X:\{X\}\to {\bf Hom}_\mathcal{C}(X,X).$$ We will abuse  notation and denote the image of $X$ under $1_X$ by $1_X$.
\item For each triplet of objects $X,Y,Z\in{\bf Ob}_\mathcal{C}$, a {\bf composition function} $$\circ_{XYZ}:{\bf Hom}_\mathcal{C}(Y,Z)\times{\bf Hom}_\mathcal{C}(X,Y)\to{\bf Hom}_\mathcal{C}(X,Z).$$
\end{enumerate}
These data are subject to the usual axioms for a category (with the variables/parameters interpreted as collections). For each $n<\omega$, we say that a category $\mathcal{C}$ is
\begin{itemize}
\item {\bf $n$-small} iff ${\bf Ob}_\mathcal{C}\in\widehat{V_n}$ and ${\bf Hom}_\mathcal{C}\in\widehat{V_n}$,
\item {\bf locally $n$-small} iff ${\bf Ob}_\mathcal{C}\in\widehat{V_{n+1}}$ and ${\bf Hom}_\mathcal{C}(X,Y)\in\widehat{V_n}$ for all $X,Y\in{\bf Ob}_\mathcal{C}$,
\item {\bf $n$-tiny} iff $\mathcal{C}$ is $m$-small for some $m<n$, {\bf $n$-large} iff $\mathcal{C}$ is $n+1$-small but not $n$-small, and {\bf $n$-very large} iff $\mathcal{C}$ is not $n+1$-small.
\end{itemize}
We define functors etc. in the obvious way relative to this definition for categories. We say that a category $\mathcal{C}$ is {\bf $n$-complete} iff it has limits for all diagrams of shape $\mathcal{S}$ where $\mathcal{S}$ is any $n$-small category. \\
\end{defn}

Freyd originally showed that any 'large' category $\mathcal{C}$ with all 'large' limits is necessarily thin, meaning $\mathcal{C}$ has at most one arrow between any two objects. The 'large' categories considered by Freyd are $1$-complete $1$-small ($0$-large) categories in our theory, and Freyd's observation trivially generalizes to the following. \\

\begin{lem}[{\it Freyd}]
Any $n$-complete $n$-small category is thin.
\end{lem}
\begin{proof}
Let $\mathcal{C}$ be an $n$-small $n$-complete category, and suppose we have objects $X,Y\in{\bf Ob}_\mathcal{C}$ with arrows $f,g:X\rightrightarrows Y$. Consider the limit $$\Big(\prod_{{\bf Hom}_\mathcal{C}} Y,\{\pi_h:\prod_{{\bf Hom}_\mathcal{C}}Y\to Y\}_{h\in{\bf Hom}_\mathcal{C}}\Big).$$ $f\neq g$ implies the existence of $2^{{\bf Hom}_\mathcal{C}}$ arrows $X\to\prod_{{\bf Hom}_\mathcal{C}}Y$ in ${\bf Hom}_\mathcal{C}$, so $$|\mathcal{P}({\bf Hom}_\mathcal{C})|=|2^{{\bf Hom}_\mathcal{C}}|\leq|{\bf Hom}_\mathcal{C}|$$ contradicting Cantors lemma, thus $f=g$.  \\
\end{proof}

The proof in this setting lays bare the real issue at hand, our ability to index a limit in a category over the collection of all arrows in that category. We can allow all limits for collections of sizes 'smaller' than the collection of all arrows in a category without issue, so $n$-large locally $n$-small categories can be $n$-complete without falling prey to the above lemma. \\

\section{Categories of Collections}

In this theory we have a hierarchy of 'categories of collections', the bottom of which is the standard category of sets.  \\

\begin{defn}[{\it Category of Collections}]
For each natural number $n$, we define a category $${\bf Coll}_n$$ called the {\bf category of $n$-collections} as follows. 
\begin{enumerate}
\item The objects of ${\bf Coll}_n$ are $n$-collections. That is, $${\bf Ob}_{{\bf Coll}_n}=\widehat{V_n}.$$
\item The arrows of ${\bf Coll}_n$ are functions between $n$-collections. That is, for $n$-collections $X$ and $Y$ we have $${\bf Hom}_{{\bf Coll}_n}(X,Y)=\{f:X\to Y|f\ \text{is a function}\}.$$
\item Identities are given by identity functions.
\item Composition is given by composition of functions.
\end{enumerate} 
For each $n<\omega$, we define a functor $$\mathcal{H}_n:{\bf Coll}_n\to{\bf Coll}_{n+1}$$ given by sending collections and functions to themselves. That is, $$\mathcal{H}_n(X)=X,$$ $$\mathcal{H}_n(f:X\to Y)=f:X\to Y.$$ We will refer to $\mathcal{H}_n$ as the {\bf $n^{th}$ collection hierarchy embedding.} We define $${\bf Set}:={\bf Coll}_0,$$ and refer to ${\bf Set}$ as the {\bf category of sets}. Further, we define $${\bf Class}:={\bf Coll}_1$$ and refer to {\bf Class} as the {\bf category of classes}. \\
\end{defn}

In what follows, if we state a definition or theorem involving ${\bf Coll}_n$ without reference to any specific index we mean the definition/theorem to be universally quantified, i.e. applying for all $n<\omega$. The category {\bf Class} already doesn't exist in $ZFC/GBC/MK/$etc., since any class that is a member of another class in these theories is immediately a set. In $HOST$, we are allowed to put classes in other classes (and $n+1$-collections in other $n+1$-collections more generally) without immediately causing them to be sets ($n$-collections) -- the only $n+1$-collections that 'become $n$-collections' are those that fall under the scope of replacement. \\

\begin{lem}[{\it Hierarchy of Categories of Collections}]
${\bf Coll}_n$ is an $n$-large locally $n$-small $n$-complete topos. Further, for all $n<\omega$ the $n^{th}$ collection hierarchy embedding $$\mathcal{H}_{n}:{\bf Coll}_n\hookrightarrow{\bf Coll}_{n+1}$$ is full, faithful, and strictly creates and preserves $n$-(co)limits.
\end{lem}
\begin{proof}
These are all straightforward observations. For details on completeness, being a topos, and $n$-largeness/local $n$-smallness, see any proof that ${\bf Set}$ has these properties but interpret the symbols for sets as representing $n$-collections instead. Faithfulness for $\mathcal{H}_n$ is trivial since it is the identity on its domain, and fullness holds since if $X$ and $Y$ are $n$-collections then so is any function $f:X\to Y\in\mathcal{P}(X\times Y)$. All other claims are similarly trivial to see, noting that we may have $n+1$-(co)limits that are not created by $\mathcal{H}_n$. \\
\end{proof}

We could obviously assume {\bf A6}, define ${\bf Coll}_\omega$, and observe that the colimit of $$\xymatrix@C5mm{{\bf Coll}_0 \ar[r]^{\mathcal{H}_0} & {\bf Coll}_1 \ar[r]^{~\;\mathcal{H}_1} & \cdots}$$ is ${\bf Coll}_\omega$ together with the obvious subcategory inclusions $\mathcal{H}_{n,\omega}:{\bf Coll}_n\to{\bf Coll}_\omega$. \\

\section{Categories of Categories}

And we finally come to the big shlamozzle in any discussion of foundations for category theory -- what do we use for a 'category of categories'? \\

\begin{defn}[{\it $2$-Category of $n$-Small Categories}]
For all $n<\omega$, we define a $2$-category $$\mathfrak{cat}_n$$ called the {\bf $2$-category of $n$-small categories} as follows: 
\begin{enumerate}
\item The objects of $\mathfrak{cat}_n$ are $n$-small categories. That is, $${\bf Ob}_{\mathfrak{cat}_n}={^{n+1}}\{\mathcal{C}|\ \mathcal{C}\ \text{is a category}\}.$$
\item For each pair of $n$-small categories $\mathcal{C}$, $\mathcal{D}$, the component category $$\mathfrak{cat}_n(\mathcal{C},\mathcal{D})$$ is given by the category of functors and natural transformations from $\mathcal{C}$ to $\mathcal{D}$. That is, $$\mathfrak{cat}_n(\mathcal{C},\mathcal{D})=\mathcal{D}^\mathcal{C}.$$ Recall that $\mathcal{D}^\mathcal{C}$ is defined as follows:
\begin{itemize}
\item The objects of $\mathcal{D}^\mathcal{C}$ are functors $F:\mathcal{C}\to\mathcal{D}$. That is, $${\bf Ob}_{\mathcal{D}^\mathcal{C}}={^{n+1}}\{F:\mathcal{C}\to\mathcal{D}|F\ \text{is a functor}\}.$$ 
\item For each pair of functors $F,G:\mathcal{C}\rightrightarrows\mathcal{D}$, the hom-collection from $F$ to $G$ is given by the collection of natural transformations from $F$ to $G$. That is, $${\bf Hom}_{\mathcal{D}^\mathcal{C}}(F,G)={^{n+1}}\{\alpha:F\Rightarrow G|\alpha\ \text{is a natural transformation}\}.$$ 
\end{itemize}
\item Identities are given by identity functors and identity natural transformations.
\item Composition functors are given by composition of functors and Godement products of natural transformations.
\end{enumerate}
\end{defn}
\begin{lem}
$\mathfrak{cat}_n$ is an $n$-large locally $n$-small category.
\end{lem}
\begin{proof}
$\mathfrak{cat}_n$ is $n+1$-small by definition. If $\mathfrak{cat}_n$ were $n$-small we would have $\mathfrak{cat}_n\in{\bf Ob}_{\mathfrak{cat}_n}$, a contradiction by foundation, so $\mathfrak{cat}_n$ is not $n$-small. If $\mathcal{C}$ and $\mathcal{D}$ are $n$-small categories then so is $\mathcal{D}^\mathcal{C}$, since for all $n$-collections $X,Y,$ and functions $f:X\to Y$ we have $f:X\to Y\in\mathcal{P}(X\times Y),$ so $$Y^X={^{n+1}}\{f:X\to Y|f\ \text{is a function}\}\subseteq\mathcal{P}(X\times Y)\in\widehat{V_n}\implies Y^X\in\widehat{V_n}.$$ Recalling that $\mathcal{C}$ and $\mathcal{D}$ are $n$-small we have that $${\bf Ob}_{\mathcal{D}^\mathcal{C}}\subseteq{\bf Ob}_\mathcal{D}^{{\bf Ob}_\mathcal{C}}\times{\bf Hom}_\mathcal{D}^{{\bf Hom}_\mathcal{C}}\in\widehat{V_n}\implies{\bf Ob}_{\mathcal{D}^\mathcal{C}}\in\widehat{V_n}.$$ Recalling that a natural transformation is just a subset of arrows $$\alpha={^{n+1}}\{\alpha_X:F(X)\to G(X)|X\in{\bf Ob}_\mathcal{C}\}\subseteq{\bf Hom}_\mathcal{D}$$ satisfying some coherence diagrams, we have again that $\alpha\subseteq{\bf Hom}_\mathcal{D}\implies\alpha\in\mathcal{P}({\bf Hom}_\mathcal{D})$, so $${\bf Hom}_{\mathcal{D}^\mathcal{C}}\subseteq\mathcal{P}({\bf Hom}_\mathcal{D})\in\widehat{V_n}\implies{\bf Hom}_{\mathcal{C}^\mathcal{D}}\in\widehat{V_n},$$ so $\mathcal{D}^\mathcal{C}$ is $n$-small. Since $\mathcal{C}$ and $\mathcal{D}$ were arbitrary $n$-small categories, $\mathfrak{cat}_n$ is locally $n$-small. \\
\end{proof}

We can collect up all locally $n$-small categories by moving up two levels. \\

\begin{defn}[{\it $2$-Category of Locally $n$-Small Categories}]
For all $n<\omega$, we define a $2$-category $$\mathfrak{Cat}_n$$ called the {\bf $2$-category of locally $n$-small categories} as follows:
\begin{itemize}
\item The objects of $\mathfrak{Cat}_n$ are locally $n$-small categories. That is, $${\bf Ob}_{\mathfrak{Cat}_n}={^{n+2}}\{\mathcal{C}|\ \mathcal{C}\ \text{is a locally $n$-small category}\}.$$
\item The component categories and remaining data are defined as in $\mathfrak{cat}_n$, moving up two stages when using collection building since the object collections are only $n+1$-collections a-priori. \\
\end{itemize}
\end{defn}
\begin{lem}
$\mathfrak{Cat}_n$ is an $n+2$-small category that is locally $n+1$-small.
\end{lem}
\begin{proof}
$\mathfrak{Cat}_n$ is $n+2$-small by definition. Local $n+1$-smallness follows from arguments similar to those in the preceding lemma (function spaces between $n+1$-collections are again $n+1$-collections). \\
\end{proof}

We can also collect up all $n$-large categories by moving two levels up. \\

\begin{defn}[{\it $2$-Category of $n$-Large Categories}]
For all $n<\omega$, we define a $2$-category $$\text{\textgoth{CAT}}_n$$ called the {\bf $2$-category of $n$-large categories} as follows:
\begin{enumerate}
\item The objects of \textgoth{CAT}$_n$ are $n$-large categories. That is, $${\bf Ob}_{\text{\textgoth{CAT}}_n}={^{n+2}}\{\mathcal{C}|\mathcal{C}\ \text{is an $n$-large category}\}.$$
\item The component categories and remaining data are defined as in $\mathfrak{Cat}_n$. \\
\end{enumerate}
\end{defn}

In general we have that $\mathfrak{cat}_n\nsubseteq\text{\textgoth{CAT}}_n\subset\mathfrak{cat}_{n+1}$ since $n$-small categories are excluded from \textgoth{CAT}$_n$ by definition. We do have that $\mathfrak{cat}_n\cup\text{\textgoth{CAT}}_n=\mathfrak{cat}_{n+1}$, and further for all $n<\omega$ $$\mathfrak{cat}_n\subset\mathfrak{Cat}_n\subset\mathfrak{cat}_{n+1}$$ by definition, so in particular $$\mathfrak{cat}_0\subset\mathfrak{Cat}_0\subset\mathfrak{cat}_1\subset\cdots.$$ Even more strongly, for all $n<\omega$ $$\mathfrak{cat}_n\in\mathfrak{Cat}_n\in\mathfrak{Cat}_{n+1}$$ by the above lemmas, so in particular $$\mathfrak{cat}_0\in\mathfrak{Cat}_0\in\mathfrak{Cat}_1\in\cdots.$$  Once again, we could assume {\bf A6}, define $\mathfrak{cat}_\omega$ and $\mathfrak{Cat}_\omega$, and observe that the former is the colimit of the inclusion functors induced by the above chain of inclusions, the former includes and contains all preceding categories, and the latter includes and contains the former. Higher $n$-categories (in the categorical sense) admit multiple possible levels of 'local smallness', and all of them can be collected up using sufficiently higher universes of collections. For $\infty$-categories we need to add {\bf A6} and move up to $GHOST$ to define the final stage of 'local smallness', and this is the only categorical setting where $GHOST$ is easier to use than $HOST$ that the author can see. \\

\chapter{References}

\begin{itemize}
\item $[\text{Kam19}]$ - Williams, Kameryn. "{\it Minimum Models of Second Order Set Theories}". The Journal of Symbolic Logic, 84(2), 589-620. \href{https://doi.org/10.1017/jsl.2019.27}{doi:10.1017/jsl.2019.27}
\item $[\text{Zer08}]$ - Zermelo, Ernst. "{\it Untersuchungen über die Grundlagen der Mengenlehre I}". Mathematische Annalen, 65 (2): 261–281, \href{https://doi.org/10.1007/BF01449999}{doi:10.1007/bf01449999}, S2CID 120085563
\item $[\text{Fra21}]$ - Fraenkel, Abraham. "{\it Uber die Zermelosche Begründung der Mengenlehre}". Jahresbericht der Deutschen Mathematiker-Vereinigung. vol. 30. pp. 97-98.
\item $[\text{Mon69}]$ - Monk, J. Donald. "Introduction to Set Theory." McGraw-Hill Book Company, 1969. \\ \href{http://euclid.colorado.edu/~monkd/monk11.pdf}{http://euclid.colorado.edu/~monkd/monk11.pdf}. \\
\end{itemize}

\end{document}